\documentclass[a4paper,12pt]{article}

\usepackage{amsmath,amssymb,amsthm}
\usepackage{latexsym}
\usepackage{graphics}
\usepackage{cite}
\usepackage{graphicx,psfrag}
\usepackage[bf, small]{titlesec}
\usepackage{authblk} 

\theoremstyle{plain}
\newtheorem{Thm}{Theorem}[section]
\newtheorem{Lem}[Thm]{Lemma}
\newtheorem{Prop}[Thm]{Proposition}
\newtheorem{Cor}[Thm]{Corollary}

\theoremstyle{definition}
\newtheorem{Defi}[Thm]{Definition}
\newtheorem{Rem}[Thm]{Remark}
\newtheorem{Ex}[Thm]{Example}

\title{On the competition graphs of $d$-partial orders}

\author[1]{Jihoon CHOI
\thanks{This research was supported by Global Ph.D Fellowship Program through the National Research Foundation of Korea(NRF) funded by the Ministry of Education (No.\ NRF-2015H1A2A1033541).}}
\author[2]{Kyeong Seok KIM
\thanks{This work was partially supported by the Korea Foundation
for the Advancement of Science \& Creativity (KOFAC) grant
funded by the Korean Government (MEST).}}
\author[1]{Suh-Ryung KIM\textsuperscript{$\ast$}
\thanks{This work was partially supported by
the National Research Foundation of Korea (NRF) grant funded
by the Korea government (MEST) (No.\ NRF-2015R1A2A2A01006885).}
\thanks{Corresponding author: srkim@snu.ac.kr}}
\author[1]{Jung Yeun LEE}
\author[3]{Yoshio SANO
\thanks{This work was supported by JSPS KAKENHI grant number 15K20885.
}}

\affil[1]{Department of Mathematics Education,
Seoul National University,
Seoul 151-742, \newline
Republic of Korea}
\affil[2]{Department of Mathematical Sciences,
Korea Advanced Institute of Science and Technology, \newline
Daejeon 305-701, Republic of Korea}
\affil[3]{Division of Information Engineering,
Faculty of Engineering, Information and Systems, \newline
University of Tsukuba, Ibaraki 305-8573, Japan}

\date{}

\begin{document}

\maketitle

\begin{abstract}
In this paper, we study the competition graphs of $d$-partial orders
and obtain their characterization which extends results given
by Cho and Kim [H.~H.~Cho and S.~-R.~Kim:
A class of acyclic digraphs with interval competition graphs,
\emph{Discrete Applied Mathematics}
\textbf{148} (2005) 171--180].
We also show that any graph can be made into the competition graph of a $d$-partial order
for some positive integer $d$
as long as adding isolated vertices is allowed.
We then study graphs whose partial order competition dimensions are at most three,
where the partial order competition dimension of a graph $G$
is the smallest nonnegative integer $d$ such that $G$ together with some isolated vertices
is the competition graph of a $d$-partial order.
\end{abstract}




\section{Introduction}

The \emph{competition graph} of a digraph $D$, which is denoted by $C(D)$,
has the same vertex set as $D$ and has an edge $xy$
between two distinct vertices $x$ and $y$
if for some vertex $z \in V$,
the arcs $(x,z)$ and $(y,z)$ are in $D$.
The competition graph has been extensively studied over the last 40 years
(see the survey articles by Kim~\cite{Kim93} and Lundgren~\cite{Lundgren89}).

Let $d$ be a positive integer.
For $\mathbf{x} = (x_1,x_2,\ldots, x_d)$,
$\mathbf{y} = (y_1,y_2,\ldots, y_d) \in \mathbb{R}^d$,
we write
$\mathbf{x} \prec \mathbf{y}$
if $x_i<y_i$ for each $i=1, \ldots, d$.
For a finite subset $S$ of $\mathbb{R}^d$,
let $D_S$ be the digraph defined by $V(D_S) = S$ and
$A(D_S) = \{(\mathbf{x},\mathbf{v}) \mid \mathbf{v}, \mathbf{x} \in S,
\mathbf{v} \prec \mathbf{x} \}$.
A digraph $D$ is called a \emph{$d$-partial order}
if there exists a finite subset $S$ of $\mathbb{R}^d$
such that $D$ is isomorphic to the digraph $D_S$.
By convention, the zero-dimensional Euclidean space $\mathbb{R}^0$
consists of a single point $0$.
In this context, we define a digraph with exactly one vertex as a \emph{$0$-partial order}.
A $2$-partial order is also called \emph{doubly partial order}
(see Figure~\ref{egofdpo} for an example).

\begin{figure}
\psfrag{D}{\small $D$}
\psfrag{u(1,6)}{\small $(1,6)$} \psfrag{h(4,0)}{\small $(4,0)$}
\psfrag{v(2,5)}{\small $(2,5)$} \psfrag{w(4,4)}{\small $(4,4)$}
\psfrag{p(3,3)}{\small $(3,3)$} \psfrag{q(3,1)}{\small $(3,1)$}
\psfrag{s(1,2)}{\small $(1,2)$} \psfrag{t(0,4)}{\small $(0,4)$}
\psfrag{x(4,3)}{\small $(4,3)$} \psfrag{y(5,2)}{\small $(5,2)$}
\psfrag{z(5,1)}{\small $(5,1)$}
\begin{center}
\includegraphics[height=5cm]{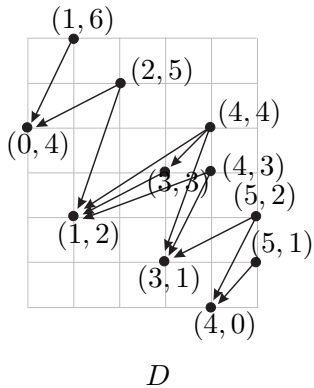}
\end{center}
\caption{A doubly partial order $D$}
\label{egofdpo}
\end{figure}

Cho and Kim~\cite{chokim} studied the competition graphs of doubly partial orders
and showed that the competition graphs of these digraphs are interval graphs
and any interval graph can be made into the competition graph of a doubly partial order
by adding sufficiently many isolated vertices.

\begin{Thm}[\cite{chokim}]\label{dpo}
The competition graph of a doubly partial order is an interval graph.
\end{Thm}

\begin{Thm}[\cite{chokim}]\label{interval}
Every interval graph can be made into the competition graph of
a doubly partial order by adding sufficiently many isolated vertices.
\end{Thm}

\noindent
Several variants of competition graphs of doubly partial orders
also have been studied
(see \cite{SJkim, Niche2009, hypergraph, LuWu, PLK:mStepDPO, Phylogeny}).

In this paper, we study the competition graphs of $d$-partial orders.
We obtain their characterization which nicely extends results given by Cho and Kim~\cite{chokim}.
We also show that any graph can be made into the competition graph of a $d$-partial order
for some positive integer $d$
as long as adding isolated vertices is allowed.
We then introduce the notion of the partial order competition dimension of a graph.
Especially, we study graphs whose partial order competition dimensions
are at most three.

\section{The competition graphs of $d$-partial orders}

In this section, we use the following notation.
We use a bold faced letter to represent a point
in $\mathbb{R}^d$ ($d \geq 2$).
For $\mathbf{x} \in \mathbb{R}^d$,
let $x_i$ denote the $i$th component of $\mathbf{x}$
for each $i=1, \ldots, d$.
Let $\mathbf{e}_i \in \mathbb{R}^d$
be the standard unit vector whose $i$th component is $1$, i.e.,
$\mathbf{e}_1:=(1, 0, \ldots, 0)$, $\ldots$, $\mathbf{e}_d:=(0, \ldots, 0,1)$.
Let $\mathbf{1}$ be the all-one vector $(1,\ldots,1)$ in $\mathbb{R}^d$.
Note that, for $\mathbf{x} \in \mathbb{R}^d$,
the standard inner product of $\mathbf{x}$ and $\mathbf{1}$ is
\[
\mathbf{x} \cdot \mathbf{1} = \sum_{i=1}^d x_i.
\]

For $\mathbf{v}_1, \ldots, \mathbf{v}_n \in \mathbb{R}^d$,
let $\text{{\rm Conv}}(\mathbf{v}_1, \ldots, \mathbf{v}_n)$
denote the \emph{convex hull} of $\mathbf{v}_1, \ldots, \mathbf{v}_n \in \mathbb{R}^d$, i.e.,
\[
\text{{\rm Conv}}(\mathbf{v}_1, \ldots, \mathbf{v}_n) :=
\left\{ \sum_{i=1}^n \lambda_i\mathbf{v}_i \mid \sum_{i=1}^n \lambda_i=1,
\lambda_i \geq 0, 1
\leq i \leq n \right\}.
\]

\subsection{The regular $(d-1)$-dimensional simplex \boldmath$\triangle^{d-1}{(\bf p)}$}

Let $\mathcal{H}^d$ be
the hyperplane in $\mathbb{R}^d$ defined by
the equation $\mathbf{x} \cdot \mathbf{1} = 0$,
and let $\mathcal{H}_+^d$ be
the open half space in $\mathbb{R}^d$ defined by
the inequality $\mathbf{x} \cdot \mathbf{1} > 0$,
i.e.,
\[
\mathcal{H}^d := \{\mathbf{x} \in \mathbb{R}^d \mid
\mathbf{x} \cdot \mathbf{1} = 0 \},
\qquad
\mathcal{H}_+^d := \{\mathbf{x} \in \mathbb{R}^d \mid
\mathbf{x} \cdot \mathbf{1} > 0 \}.
\]
We fix a point $\mathbf{p}$ in $\mathcal{H}_+^d$.
Let $\triangle^{d-1}(\mathbf{p})$ be the intersection of the hyperplane $\mathcal{H}^d$
and the closed cone
\[
\{\mathbf{x} \in \mathbb{R}^d \mid x_i \leq p_i \ (i= 1, \ldots, d) \}.
\]

\begin{figure}
\psfrag{p}{\footnotesize${\bf p}=(p_1,p_2,p_3)$}
\psfrag{q}{\footnotesize$\mathcal{H}^3$}
\psfrag{x}{\footnotesize$x$}
\psfrag{y}{\footnotesize$y$}
\psfrag{z}{\footnotesize$z$}
\psfrag{r}{\footnotesize$\triangle^2({\bf p})$}
\psfrag{s}{\footnotesize$(-p_2-p_3,p_2,p_3)$}
\psfrag{t}{\footnotesize$(p_1,-p_1-p_3,p_3)$}
\psfrag{u}{\footnotesize$(p_1,p_2,-p_1-p_2)$}
\begin{center}
\includegraphics[height=5cm]{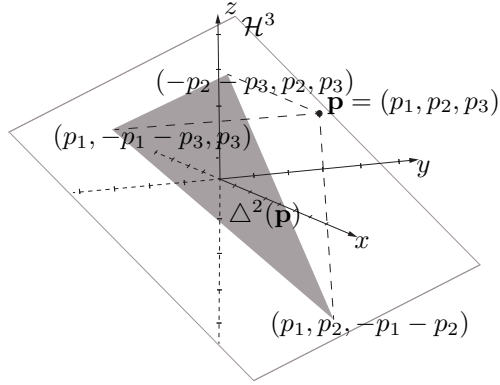}
\end{center}
\caption{A point ${\bf p} \in \mathcal{H}^3_+$ and the triangle $\triangle^2({\bf p})$}
\label{tetrahedron}
\end{figure}

\begin{Lem}\label{lem:delta}
For $\mathbf{p} \in \mathcal{H}_+^d$,
the set $\triangle^{d-1}(\mathbf{p})$ is the convex hull
$\text{{\rm Conv}}(\mathbf{v}_1, \ldots, \mathbf{v}_d)$
of the vectors $\mathbf{v}_1, \ldots, \mathbf{v}_d$ defined by
\begin{equation*}
\mathbf{v}_ i = \mathbf{p} - (\mathbf{p} \cdot \mathbf{1}) \mathbf{e}_i
\quad (i = 1, \ldots, d).
\end{equation*}
Moreover, $\triangle^{d-1}(\mathbf{p})$ is a regular $(d-1)$-simplex.
\end{Lem}

\begin{proof}
Let $\mathbf{v}_1$, $\ldots$, $\mathbf{v}_d \in \mathbb{R}^d$ be the intersections of
the hyperplane $\mathcal{H}^d$
and the lines going through $\mathbf{p}$ with directional vectors
$\mathbf{e}_1, \ldots, \mathbf{e}_d$, respectively.
Then
$\mathbf{v}_ i= \mathbf{p} - (\mathbf{p} \cdot \mathbf{1}) \mathbf{e}_i$ for $i = 1, \ldots, d$.
By definition, we have
$\triangle^{d-1}(\mathbf{p}) = \text{{\rm Conv}}(\mathbf{v}_1, \ldots, \mathbf{v}_d)$.
Since $\mathbf{v}_1, \ldots, \mathbf{v}_d$ are linearly independent,
the set $\triangle^{d-1}(\mathbf{p})$ is a $(d-1)$-simplex.
Moreover, since the length of each edge of $\triangle^{d-1}(\mathbf{p})$ is equal to
$\sqrt{2}(\mathbf{p} \cdot \mathbf{1})$,
the simplex $\triangle^{d-1}(\mathbf{p})$ is regular.
\end{proof}

Note that the distance between $\mathbf{p}$ and each vertex of $\triangle^{d-1}(\mathbf{p})$ is
equal to $\mathbf{p} \cdot \mathbf{1}$.
Moreover, the directional vector for the line passing through the vertices $\mathbf{v}_i$ and $\mathbf{v}_j$ is
$\mathbf{e}_j - \mathbf{e}_i$
for distinct $i$, $j$ in $\{1,\ldots,d\}$.
The center of $\triangle^{d-1}(\mathbf{p})$ is
$\frac{1}{d} \sum_{i=1}^d \mathbf{v}_i
= \mathbf{p} - \frac{1}{d}(\mathbf{p} \cdot \mathbf{1}) \mathbf{1}$.
Therefore,
the directional vector from this center to the point $\mathbf{p}$ is parallel to the all-one vector $\mathbf{1}$, and
the distance between this center and the point $\mathbf{p}$ is $\frac{1}{\sqrt{d}}(\mathbf{p} \cdot \mathbf{1})$
which is $\frac{1}{\sqrt{2d}}$ times the edge length of $\triangle^{d-1}(\mathbf{p})$.

We say that two geometric figures in $\mathbb{R}^d$ are \emph{homothetic} if
they are related by a geometric contraction or expansion.
From the above observation, we can conclude the following:

\begin{Prop}\label{X1}
If $\mathbf{p}, \mathbf{q} \in \mathcal{H}_+^d$,
then $\triangle^{d-1}(\mathbf{p})$ and $\triangle^{d-1}(\mathbf{q})$ are homothetic.
\end{Prop}

\subsection{A bijection from $\mathcal{H}_+^d$
to the set of certain regular $(d-1)$-simplices}

\begin{Lem}\label{fixed}
The vertices of $\triangle^{d-1}(\mathbf{1})$
may be labeled as $\mathbf{w}_1, \ldots, \mathbf{w}_d$
so that $\mathbf{w}_j-\mathbf{w}_i$ is a positive scalar multiple of
$\mathbf{e}_i - \mathbf{e}_{j}$ for any distinct $i,j$ in $\{1,\ldots,d\}$.
\end{Lem}

\begin{proof}
By Lemma~\ref{lem:delta}, the vertices of $\triangle^{d-1}(\mathbf{1})$ are
$\mathbf{1} - d \mathbf{e}_{i}$ for $i=1, \ldots, d$.
We denote $\mathbf{1} - d \mathbf{e}_{i}$ by $\mathbf{w}_i$ to obtain the desired labeling.
\end{proof}

\begin{Lem}\label{regsim}
Let $d$ be an integer with $d \geq 2$.
Suppose that
$\Lambda$ is a regular $(d-1)$-simplex contained in the hyperplane $\mathcal{H}^d$
homothetic to $\triangle^{d-1}(\mathbf{1})$.
Then,
there exists
$\mathbf{p} \in \mathcal{H}_+^d$
such that $\Lambda=\triangle^{d-1}(\mathbf{p})$.
\end{Lem}

\begin{proof}
Since $\Lambda$ is homothetic to $\triangle^{d-1}(\mathbf{1})$,
there exists $\mathbf{v}_1, \ldots, \mathbf{v}_d \in \mathbb{R}^d$
which are linearly independent
such that
$\Lambda=\text{{\rm Conv}}(\mathbf{v}_1, \ldots, \mathbf{v}_d)$
and $\mathbf{v}_j - \mathbf{v}_i$ is a positive scalar multiple of
$\mathbf{e}_i - \mathbf{e}_j$ for any distinct $i$ and $j$ in $\{1, \ldots, d\}$ by Lemma~\ref{fixed}.
Moreover, ${\mathbf v}_i \cdot \mathbf{e}_i  < {\mathbf v}_j \cdot \mathbf{e}_i $
for any distinct $i$ and $j$ in $\{1, \ldots, d\}$.
Let $\mathbf{p} \in \mathbb{R}^d$ be a vector
defined by $p_i : \max\{ \mathbf{v}_k \cdot \mathbf{e}_i \mid 1 \leq k \leq d \}$
$(i=1, \ldots, d)$.
Then, $\Lambda=\triangle^{d-1}(\mathbf{p})$.
Since $\mathbf{v}_1 \in \mathcal{H}^d$,
we have
$\mathbf{v}_1 \cdot \mathbf{1}=0$.
Since $p_i \geq \mathbf{v}_1 \cdot \mathbf{e}_i$ for any $i=1, \ldots, d$,
we have
$\mathbf{p} \cdot \mathbf{1} \geq \mathbf{v}_1 \cdot \mathbf{1}=0$.
If $\mathbf{p} \cdot \mathbf{1} = 0$, then
we obtain $\mathbf{p} = \mathbf{v}_1$ and
$p_1 = \mathbf{v}_1 \cdot \mathbf{e}_1 < \mathbf{v}_2 \cdot \mathbf{e}_1$,
which is a contradiction to the definition of $p_1$.
Therefore $\mathbf{p} \cdot \mathbf{1} > 0$,
i.e., $\mathbf{p} \in \mathcal{H}_+^d$.
Thus the lemma holds.
\end{proof}

Let $d$ be an integer with $d \geq 2$.
Let $\mathcal{F}^{d-1}_*$ be
the set of
the regular $(d-1)$-simplices in
the hyperplane $\mathcal{H}^d$ which are homothetic to $\triangle^{d-1}(\mathbf{1})$.
Let $f_*$ be a map from $\mathcal{H}_+^d$ to $\mathcal{F}^{d-1}_*$
defined by $f_*(\mathbf{p}) = \triangle^{d-1}(\mathbf{p})$.
By Lemma \ref{lem:delta} and Proposition~\ref{X1},
$\triangle^{d-1}(\mathbf{p}) \in \mathcal{F}^{d-1}_*$ and therefore
the map $f_*$ is well-defined.

\begin{Prop}\label{prop:bijection}
For each integer $d \geq 2$,
the map $f_* : \mathcal{H}_+^d \to \mathcal{F}^{d-1}_*$ is a bijection.
\end{Prop}

\begin{proof}
By Lemma \ref{regsim}, the map $f_*$ is surjective.
Suppose that $\triangle^{d-1}(\mathbf{p}) = \triangle^{d-1}(\mathbf{q})$.
Since the centers of $\triangle^{d-1}(\mathbf{p})$ and $\triangle^{d-1}(\mathbf{q})$
are the same, we have
$\mathbf{p} - \frac{1}{d}(\mathbf{p} \cdot \mathbf{1}) \mathbf{1}
= \mathbf{q} - \frac{1}{d}(\mathbf{q} \cdot \mathbf{1}) \mathbf{1}$.
Since the lengths of edges of $\triangle^{d-1}(\mathbf{p})$ and $\triangle^{d-1}(\mathbf{q})$
are the same, we have
$\sqrt{2}(\mathbf{p} \cdot \mathbf{1}) = \sqrt{2}(\mathbf{q} \cdot \mathbf{1})$.
Therefore, we have $\mathbf{p} = \mathbf{q}$.
Thus the map $f_*$ is injective.
Hence the map $f_*$ is a bijection.
\end{proof}

Let $\mathcal{F}^{d-1}$ be the set of the interiors of the regular $(d-1)$-simplices in
the hyperplane $\mathcal{H}^d$ which are homothetic to $\triangle^{d-1}(\mathbf{1})$.
Then there is a clear bijection $\varphi: \mathcal{F}_*^{d-1} \rightarrow \mathcal{F}^{d-1}$
such that for each element in $\mathcal{F}_*^{d-1}$, its $\varphi$-value is its interior.
Therefore we obtain the following corollary.

\begin{Cor}\label{cor:bijection}
For each integer $d \geq 2$,
the map $\varphi \circ f_* : \mathcal{H}_+^d \to \mathcal{F}^{d-1}$ is a bijection.
\end{Cor}

\subsection{A characterization of the competition graphs of $d$-partial orders}

Let $A^{d-1}(\mathbf{p})$ be the interior of the regular simplex $\triangle^{d-1}(\mathbf{p})$,
i.e., $A^{d-1}(\mathbf{p}) := \text{int}(\triangle^{d-1}(\mathbf{p}))$. Then
\[
A^{d-1}(\mathbf{p}) = \{\mathbf{x} \in \mathcal{H}^d \mid \mathbf{x} \prec \mathbf{p} \}.
\]

\begin{Prop}\label{X3}
For $\mathbf{p},\mathbf{q} \in \mathcal{H}_+^d$,
$\triangle^{d-1}(\mathbf{p})$ is
contained in $A^{d-1}(\mathbf{q})$
if and only if $\mathbf{p} \prec \mathbf{q}$.
\end{Prop}

\begin{proof}
Suppose that $\mathbf{p}\prec \mathbf{q}$.
Take a point $\mathbf{a}$ in $\triangle^{d-1}(\mathbf{p})$.
Then $a_k \leq p_k$ for each $k=1, \ldots, d$.
By the assumption that $\mathbf{p}\prec \mathbf{q}$, $a_k < q_k$ for each $k=1, \ldots, d$,
that is, $\mathbf{a} \prec \mathbf{q}$.
Thus $\triangle^{d-1}(\mathbf{p})$ is contained in $A^{d-1}(\mathbf{q})$.

Suppose that $\triangle^{d-1}(\mathbf{p})$ is
contained in $A^{d-1}(\mathbf{q})$.
Then, by Lemma~\ref{lem:delta},
$\mathbf{p} - (\mathbf{p} \cdot \mathbf{1}) \mathbf{e}_i$ $(i=1, \ldots, d)$
are points in $A^{d-1}(\mathbf{q})$.
By the definition of $A^{d-1}(\mathbf{q})$, we have
$\mathbf{p} - (\mathbf{p} \cdot \mathbf{1}) \mathbf{e}_i \prec \mathbf{q}$ $(i=1, \ldots, d)$,
which implies $\mathbf{p} \prec \mathbf{q}$.
\end{proof}

\begin{Lem}\label{lem:adj}
Let $d$ be a positive integer and let $D$ be a $d$-partial order.
Then, two vertices $\mathbf{v}$ and $\mathbf{w}$ of $D$ are adjacent in the competition graph of $D$
if and only if there exists a vertex $\mathbf{a}$ in $D$ such that
$\triangle^{d-1}(\mathbf{a}) \subseteq A^{d-1}(\mathbf{v}) \cap A^{d-1}(\mathbf{w})$.
\end{Lem}

\begin{proof}
By definition,
two vertices $\mathbf{v}$ and $\mathbf{w}$ are adjacent in the competition graph of $D$
if and only if
there exists a vertex $\mathbf{a}$ in $D$ such that $\mathbf{a} \prec \mathbf{v}$
and $\mathbf{a} \prec \mathbf{w}$.
By Proposition~\ref{X3},
$\mathbf{a} \prec \mathbf{v}$
and $\mathbf{a} \prec \mathbf{w}$ holds
if and only if
$\triangle^{d-1}(\mathbf{a}) \subseteq A^{d-1}(\mathbf{v})$
and $\triangle^{d-1}(\mathbf{a}) \subseteq A^{d-1}(\mathbf{w})$,
that is, $\triangle^{d-1}(\mathbf{a}) \subseteq A^{d-1}(\mathbf{v}) \cap A^{d-1}(\mathbf{w})$.
Thus the lemma holds.
\end{proof}

The following result extends Theorems~\ref{dpo} and \ref{interval}.

\begin{Thm}\label{thm:intersectiongeneral}
Let $G$ be a graph and let $d$ be an integer with $d \geq 2$.
Then, $G$ is the competition graph of a $d$-partial order
if and only if
there exists a family $\mathcal{F}$ of the interiors of regular $(d-1)$-simplices in $\mathbb{R}^d$
which are
contained in the hyperplane $\mathcal{H}^d$ and homothetic to
$A^{d-1}(\mathbf{1})$
and there exists a one-to-one correspondence between the vertex set of $G$ and $\mathcal{F}$
such that
\begin{itemize}
\item[{\rm ($\star$)}]
two vertices $v$ and $w$ are adjacent in $G$ if and only if
two elements in $\mathcal{F}$ corresponding to $v$ and $w$
have the intersection containing the closure of another element in $\mathcal{F}$.
\end{itemize}
\end{Thm}

\begin{proof}
First we show the ``only if" part.
Let $D$ be a $d$-partial order
and let $G$ be the competition graph of $D$.
Without loss of generality, we may assume that $V(D) \subseteq \mathcal{H}^d_+$
by translating each of the vertices of $D$ in the same direction and by the same amount
since the competition graph of $D$ is determined only by the adjacency among vertices of $D$.
Consequently $A^{d-1}(\mathbf{v}) \neq \emptyset$ for each vertex $\mathbf{v}$ of $D$.
Let
$\mathcal{F}=\{A^{d-1}(\mathbf{v}) \mid \mathbf{v} \in V(D)\}$
and let $f:V(G) \to \mathcal{F}$
be the map defined by
$f(\mathbf{v}) = A^{d-1}(\mathbf{v})$ for $\mathbf{v} \in V(D)$.
Note that $\mathcal{F} \subseteq \mathcal{F}^{d-1}$.
Since the map $f:V(G) \to \mathcal{F}$
is a restriction of the map $\varphi \circ f_*:V(G) \to \mathcal{F}^{d-1}$,
it follows from Corollary~\ref{cor:bijection} that
$f$ is a bijection.
By Lemma \ref{lem:adj}, the condition ($\star$) holds.

Second, we show the ``if" part.
Suppose that there exist a family $\mathcal{F}  \subseteq \mathcal{F}^{d-1}$
and a bijection $f:V(G) \to \mathcal{F}$ such that the condition ($\star$) holds.
By Corollary~\ref{cor:bijection},
each element in $\mathcal{F}$ can be represented as $A^{d-1}(\mathbf{p})$
for some $\mathbf{p} \in \mathcal{H}^d_+$.
Let $D$ be a digraph with vertex set
$V(D)=\{\mathbf{p} \in \mathbb{R}^d \mid A^{d-1}(\mathbf{p}) \in \mathcal{F}\}$
and arc set $A(D)=\{(\mathbf{p}, \mathbf{q}) \mid \mathbf{p}, \mathbf{q} \in V(D),
\mathbf{p} \neq \mathbf{q}, \triangle^{d-1}(\mathbf{q}) \subseteq A^{d-1}(\mathbf{p}) \}$.
By Proposition~\ref{X3}, $(\mathbf{p},\mathbf{q}) \in A(D)$ if and only if $\mathbf{q} \prec \mathbf{p}$,
so $D$ is a $d$-partial order.
Now, take two vertices $v$ and $w$ in $G$.
Then, by the hypothesis and above argument, $v$ and $w$ correspond to some points $\mathbf{p}$ and $\mathbf{q}$
in $\mathbb{R}^d$, respectively,
so that $v$ and $w$ are adjacent if and only if
both $A^{d-1}(\mathbf{p})$ and $A^{d-1}(\mathbf{q})$ contain the closure of an element in $\mathcal{F}$,
that is, $\triangle^{d-1}(\mathbf{r})$ for some $\mathbf{r} \in \mathbb{R}^d$.
By the definition of $D$, $(\mathbf{p},\mathbf{r}) \in A(D)$ and $(\mathbf{q},\mathbf{r}) \in A(D)$.
Consequently, $v$ and $w$ are adjacent if and only if
the corresponding vertices $\mathbf{p}$ and $\mathbf{q}$ have a common out-neighbor in $D$.
Hence $G$ is the competition graph of the $d$-partial order $D$.
\end{proof}

\subsection{Intersection graphs and the competition graphs of $d$-partial orders}

\begin{Thm}\label{thm:open}
If $G$ is the intersection graph of
a finite family of homothetic open regular $(d-1)$-simplices,
then $G$ together with sufficiently many new isolated vertices
is the competition graph of a $d$-partial order.
\end{Thm}

\begin{proof}
Let $\mathcal{A} = \{A_1, \ldots A_n \}$ be
a finite family of homothetic open regular $(d-1)$-simplices,
and let $G$ be the intersection graph of $\mathcal{A}$
with bijection $\phi:\mathcal{A} \to V(G)$.
For each distinct pair of $i$ and $j$ in $\{1, \ldots, n\}$
such that $A_i \cap A_j \neq \emptyset$,
let $B_{ij}$ be an open regular $(d-1)$-simplex homothetic $A_1$
such that the closure of $B_{ij}$ is contained in $A_i \cap A_j$.
We can take such $B_{ij}$ so that $B_{ij} \cap B_{i'j'} = \emptyset$
for distinct pairs $\{i,j\}$ and $\{i',j'\}$.
Let $\mathcal{B} = \{B_{ij} \mid i, j \in \{1, \ldots, n\},
i \neq j, A_i \cap A_j \neq \emptyset \}$.
Then the family $\mathcal{F} := \mathcal{A} \cup \mathcal{B}$
and a map $f: \mathcal{F} \to V(G) \cup \{ z_1, \ldots, z_{|\mathcal{B}|} \}$
such that $f|_{\mathcal{A}} = \phi$ and $f(\mathcal{B}) = \{ z_1, \ldots, z_{|\mathcal{B}|} \}$
satisfy the condition ($\star$) in Theorem \ref{thm:intersectiongeneral}.
Thus $G \cup \{ z_1, \ldots, z_{|\mathcal{B}|} \}$ is
the competition graph of a $d$-partial order.
\end{proof}

\begin{Lem}\label{lem:closed}
If $G$ is the intersection graph of
a finite family $\mathcal{F}$ of homothetic closed regular $(d-1)$-simplices,
then
$G$ is the intersection graph of
a finite family of homothetic open regular $(d-1)$-simplices.
\end{Lem}

\begin{proof}
Let $\mathcal{D} = \{\triangle_1, \ldots \triangle_n \}$ be
a finite family of homothetic closed regular $(d-1)$-simplices,
and let $G$ be the intersection graph of $\mathcal{D}$
with bijection $\phi:\mathcal{D} \to V(G)$.
Let
\[
\varepsilon = \frac{1}{3} \min \{ d( \triangle_i, \triangle_j) \mid
i, j \in \{1, \ldots, n\},
\triangle_i \cap \triangle_j = \emptyset \}
\]
where
$d( \triangle_i, \triangle_j) = \inf\{d({\mathbf x}, {\mathbf y})
\mid {\mathbf x} \in \triangle_i, {\mathbf y} \in \triangle_j \}$.

We make each simplex in $\mathcal{D}$ $(1+\varepsilon)$ times bigger
while the center of each simplex is fixed.
Then we take the interiors of these closed simplices.
By the choice of $\varepsilon$, the graph $G$ is
the intersection graph of the family of newly obtained open simplices.
Hence the lemma holds.
\end{proof}

\begin{Thm}\label{thm:closed}
If $G$ is the intersection graph of
a finite family $\mathcal{F}$ of homothetic closed regular $(d-1)$-simplices,
then $G$ together with sufficiently many new isolated vertices
is the competition graph of a $d$-partial order.
\end{Thm}

\begin{proof}
The theorem follows from
Lemma~\ref{lem:closed} and Theorem~\ref{thm:open}.
\end{proof}

\begin{Rem}
In the case where $d=2$,
Theorem~\ref{thm:closed} is the same as Theorem~\ref{interval}.
Due to Theorem~\ref{dpo}, the converse of Theorem \ref{thm:closed} is true for $d=2$.
In fact, we can show that the converse of Theorem~\ref{thm:open} is also true for $d=2$.
\end{Rem}

The following example shows that the converses of Theorems~\ref{thm:open}
and \ref{thm:closed}
are not true for $d = 3$.

\begin{Ex}
Let $G$ be a subdivision of $K_5$ given in Figure~\ref{counterexample}.
Then, by Theorem~\ref{thm:intersectiongeneral},
the family of homothetic equilateral triangles given in the figure
makes $G$ together with $9$ isolated vertices into the competition graph of a $3$-partial order.
However,
$G$ is not the intersection graph of any family of homothetic equilateral
closed triangles.
By Lemma \ref{lem:closed},
$G$ is not the intersection graph of any family of homothetic equilateral
open triangles, either.
\end{Ex}

\begin{proof}
Suppose that there exists a family $\mathcal{F}:=\{\triangle(v) \mid v \in V(G) \}$
of homothetic equilateral closed triangles
such that $G$ is the intersection graph of $\mathcal{F}$.
Since $v_1v_2v_3v_4v_1$ is an induced cycle in $G$,
the triangles $\triangle(v_1)$, $\triangle(v_2)$, $\triangle(v_3)$, and $\triangle(v_4)$
are uniquely located as in Figure~\ref{counterexample1}
up to the sizes of triangles.
Since the vertices $v_1$, $v_3$, and $v_4$ are neighbors of both $v_5$ and $v_7$ in $G$
whereas $v_2$ is not, and the vertices $v_5$ and $v_7$ are not adjacent in $G$,
we may conclude that the locations of $\triangle(v_5)$ and $\triangle(v_7)$
should be those for the triangles I and II given in Figure~\ref{counterexample1}.
Since the triangle II cannot have intersections with $\triangle(v_2)$ and the triangle I,
all of its sides are surrounded by $\triangle(v_1)$, $\triangle(v_2)$, $\triangle(v_3)$, and $\triangle(v_4)$.
Now, since $v_6$ is adjacent to $v_5$ and $v_7$,
$\triangle(v_6)$ must have intersections with both $\triangle(v_5)$ and $\triangle(v_7)$.
However, it cannot be done without having an intersection with
one of $\triangle(v_1)$, $\triangle(v_2)$, $\triangle(v_3)$, and $\triangle(v_4)$,
which is a contradiction to the fact that
none of $v_1$, $v_2$, $v_3$, and $v_4$ is adjacent to $v_6$ in $G$.
\end{proof}

\begin{figure}
\psfrag{a}{$v_1$}
\psfrag{b}{$v_2$}
\psfrag{c}{$v_3$}
\psfrag{d}{$v_4$}
\psfrag{e}{$v_5$}
\psfrag{f}{$v_6$}
\psfrag{g}{$v_7$}
\psfrag{A}{$A^2(\mathbf{v}_1)$}
\psfrag{B}{$A^2(\mathbf{v}_2)$}
\psfrag{C}{$A^2(\mathbf{v}_3)$}
\psfrag{D}{$A^2(\mathbf{v}_4)$}
\psfrag{E}{$A^2(\mathbf{v}_5)$}
\psfrag{F}{$A^2(\mathbf{v}_6)$}
\psfrag{G}{$A^2(\mathbf{v}_7)$}
\psfrag{h}{$G$}
\begin{center}
\includegraphics[height=5cm]{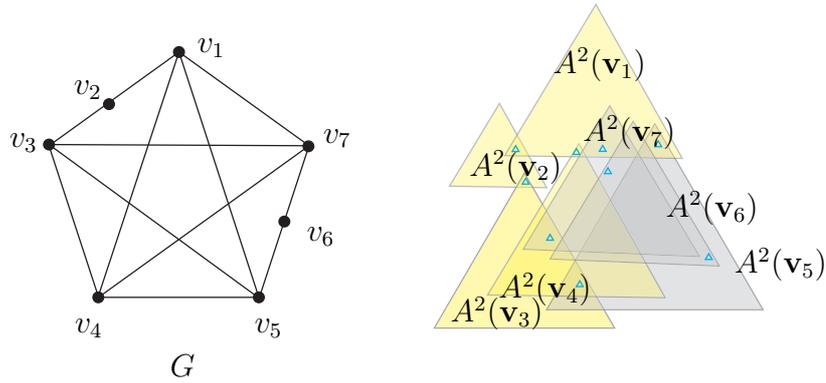}
\end{center}
\caption{A subdivision $G$ of $K_5$ and a family of homothetic equilateral triangles making
$G$ together with $9$ isolated vertices into
the competition graph of a $3$-partial order}
\label{counterexample}
\end{figure}

\begin{figure}
\psfrag{A}{$\triangle(v_1)$}
\psfrag{B}{$\triangle(v_2)$}
\psfrag{C}{$\triangle(v_3)$}
\psfrag{D}{$\triangle(v_4)$}
\psfrag{E}{I}
\psfrag{G}{II}
\begin{center}
\includegraphics[height=5cm]{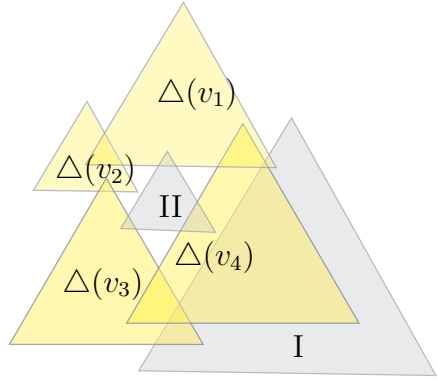}
\end{center}
\caption{An assignment of homothetic equilateral triangles to vertices
$v_1$, $v_2$, $v_3$, $v_4$, $v_5$, $v_7$ of $G$ given in Figure~\ref{counterexample}}
\label{counterexample1}
\end{figure}

\section{The partial order competition dimension of a graph}

\begin{Prop}\label{prop:01}
Let $d$ be a positive integer.
If $G$ is the competition graph of a $d$-partial order,
then $G$ is the competition graph of a $(d+1)$-partial order.
\end{Prop}

\begin{proof}
Let $D$ be a $d$-partial order such that
$G$ is the competition graph of $D$.
For each $\mathbf{v} \in V(D) \subseteq \mathbb{R}^{d}$,
we define
$\tilde{\mathbf{v}} \in \mathbb{R}^{d+1}$
by
\[
\tilde{\mathbf{v}} = \left(v_1, \ldots, v_d, \sum_{i=1}^d v_i \right).
\]
Then $\{\tilde{\mathbf{v}} \mid \mathbf{v} \in V(D) \}$
defines a $(d+1)$-partial order $\tilde{D}$.
Since
\begin{eqnarray*}
\tilde{\mathbf{v}} \prec \tilde{\mathbf{w}}
&\Leftrightarrow & {v}_i < {w}_i \ (i=1,\ldots,d) \ \mbox{ and } \  \sum_{i=1}^d v_i < \sum_{i=1}^d w_i \\
&\Leftrightarrow & {v}_i < {w}_i \ (i=1,\ldots,d) \\
&\Leftrightarrow & \mathbf{v} \prec \mathbf{w},
\end{eqnarray*}
the $(d+1)$-partial order $\tilde{D}$ is the same digraph
as the $d$-partial order $D$.
Hence $G$ is the competition graph of the $(d+1)$-partial order $\tilde{D}$.
\end{proof}

Proposition~\ref{prop:02} can be shown by using Theorem~8 in \cite{dim2poset}.
We present a new proof from which Proposition~\ref{prop:dimub} also follows.

\begin{Prop}\label{prop:02}
For any graph $G$,
there exists positive integers $d$ and $k$
such that $G$ together with $k$ isolated vertices
is the competition graph of a $d$-partial order.
\end{Prop}

\begin{proof}
Let $n = |V(G)|$ and label the vertices of $G$ as $v_1, \ldots, v_n$.
Fix four real numbers $r_1$, $r_2$, $r_3$, and $r_4$
such that $r_1 < r_2 < r_3 < r_4$.
We define a map $\phi:V(G) \to \mathbb{R}^{n}$ by
\begin{align*}
\phi(v_i)_j =
	\begin{cases}
	r_2 & \text{if } j = i; \\
	r_4 & \text{if } j \neq i.
	\end{cases}	
\end{align*}
We define a map $\psi: E(G) \to \mathbb{R}^{n}$ by
\begin{align*}
\psi(e)_k =
	\begin{cases}
	r_1 & \text{if } v_k \in e; \\
	r_3 & \text{if } v_k \notin e.
	\end{cases}	
\end{align*}
Let $V= \{\phi(v_i) \mid v_i \in V(G) \} \cup \{\psi(e) \mid e \in E(G) \} \subseteq \mathbb{R}^{n}$.
Then $V$ defines an $n$-partial order $D$.
By definition, the in-neighborhood of the vertex $\psi(e) \in V$
is $\{\phi(v_i), \phi(v_j)\}$ for an edge $e = \{v_i,v_j\}$ of $G$ and
the in-neighborhood of the vertex $\phi(v) \in V$
is the empty set for a vertex $v$.
Thus the competition graph of $D$ is $G$
together with isolated vertices as many as $|E(G)|$.
Hence, by taking $d=n$ and $k=|E(G)|$, we complete the proof.
\end{proof}
Now we may introduce the following notion. Recall that for a finite subset $S$ of $\mathbb{R}^d$,
$D_S$ is the digraph defined by $V(D_S) = S$ and
$A(D_S) = \{(\mathbf{x},\mathbf{v}) \mid \mathbf{v}, \mathbf{x} \in S,
\mathbf{v} \prec \mathbf{x} \}$.
\begin{Defi}
For a graph $G$, we define
the \emph{partial order competition dimension} $\dim_{\text{{\rm poc}}}(G)$ of $G$
as the smallest nonnegative integer $d$
such that $G$ together with $k$ isolated vertices is the competition graph of $D$
for some $d$-partial order $D$ and some nonnegative integer $k$,
i.e.,
\[
\dim_{\text{{\rm poc}}}(G) := \min \{d \in \mathbb{Z}_{\geq 0} \mid
\exists k \in \mathbb{Z}_{\geq 0}, \exists S \subseteq \mathbb{R}^d,
\text{ s.t. }
G \cup I_k = C(D_S) \},
\]
where $\mathbb{Z}_{\geq 0}$ is the set of nonnegative integers and $I_k$ is a set of $k$ isolated vertices.
\end{Defi}

\begin{Rem}
Wu and Lu~\cite{dim2poset} introduced the notion of
the \emph{dimension-$d$ poset competition number} of a graph $G$,
denoted by $\mathcal{PK}^d(G)$,
which is defined to be the smallest nonnegative integer $p$
such that $G$ together $p$ additional isolated vertices
is isomorphic to the competition graph of a poset of dimension at most $d$
if such a poset exists, and to be $\infty$ otherwise.
By using $\mathcal{PK}^d(G)$, the partial order competition dimension of $G$
can be represented as
$\dim_{\text{{\rm poc}}}(G)=\min \{d \mid \mathcal{PK}^d(G) < \infty \}$.
Therefore $\mathcal{PK}^d(G) < \infty$ implies that $\dim_{\text{{\rm poc}}}(G) \le d$.
In this respect, Proposition~3.2 and the ``if" part of Proposition~3.10
may follow from their result presenting the dimension-$d$ poset competition numbers
of a complete graph with or without isolated vertices,
which are also shown to be trivially true in this paper.
\end{Rem}

\begin{Prop}\label{prop:dimub}
For any graph $G$, we have
$\dim_{\text{{\rm poc}}}(G) \leq |V(G)|$.
\end{Prop}

\begin{proof}
The proposition follows from
the construction of a $d$-partial order
in the proof of Proposition~\ref{prop:02}.
\end{proof}

For a graph $G$, the partial order competition dimension of an induced subgraph of $G$
is less than or equal to that of $G$.
To show this, we need the following lemmas.

\begin{Lem}\label{lem:subgraph}
Let $D$ be a digraph and let $G$
be the competition graph of $D$.
Let $S$ be a set of vertices.
The competition graph of
$D[S]$ is a subgraph of $G[S]$,
where $D[S]$ and $G[S]$ mean the subdigraph of $D$
and the subgraph of $G$, respectively, induced by $S$.
\end{Lem}

\begin{proof}
Let $H$ be the competition graph of
$D[S]$. Obviously, $V(H) = S$.
Take an edge $\{u,v\}$ of $H$.
By definition, there exists a vertex $w$ in $D[S]$ such that $(u,w)$ and $(v,w)$ are arcs of $D[S]$.
Consequently, $(u,w)$ and $(v,w)$ are arcs of $D$ and so $\{u,v\}$ is an edge of $G$.
Since $u,v \in S$, $\{u,v\}$ is an edge of $G[S]$. Hence $H$ is a subgraph of $G[S]$.
\end{proof}

\begin{Lem}\label{lem:outneighbor}
Let $D$ be a transitive acyclic digraph and let $G$
be the competition graph of $D$.
For any non-isolated vertex $u$ of $G$,
there exists an isolated vertex $v$ of $G$
such that $(u,v)$ is an arc of $D$.
\end{Lem}

\begin{proof}
Take a non-isolated vertex $u$ of $G$. Since $u$ has a neighbor $w$ in $G$,
$u$ and $w$ have a common out-neighbor in $D$. Take a longest directed path in $D$ originating from $u$.
We denote by $v$ the terminal vertex of the directed path.
Since $D$ is acyclic, the out-degree of $v$ in $D$ is zero and so $v$ is
isolated in $G$.
By the hypothesis that $D$ is transitive, $(u,v)$ is an arc of $D$.
\end{proof}

\begin{Prop}\label{prop:induced}
Let $G$ be a graph and let $H$ be an
induced subgraph of $G$. Then
$\dim_{\text{{\rm poc}}}(H) \leq \dim_{\text{{\rm poc}}}(G)$.
\end{Prop}

\begin{proof}
Let $d = \dim_{\text{{\rm poc}}}(G)$.
Then, there exists a $d$-partial order $D$ whose competition graph is
the disjoint union of $G$ and a set $J$ of isolated vertices.
Let $I$ be the set of isolated vertices in $G$.
Let $S = V(H) \cup I \cup J \subseteq \mathbb{R}^d$.
Then the digraph $D_S$ is a $d$-partial order. By Lemma~\ref{lem:subgraph},
the competition graph of $D_S$ is a subgraph of
$H \cup (I \setminus V(H)) \cup J$.

Now take two adjacent vertices $\mathbf{x}$ and $\mathbf{y}$ in $H$.
Then, since they are adjacent in $G$, there exists a vertex $\mathbf{v} \in V(D)$
such that $\mathbf{v}\prec\mathbf{x}$ and $\mathbf{v}\prec \mathbf{y}$.
If $\mathbf{v}$ is isolated in $G$ or $\mathbf{v} \in J$,
then $(\mathbf{x},\mathbf{v})$ and $(\mathbf{y},\mathbf{v})$ belong to $A(D_S)$ by definition.
Suppose that $\mathbf{v} \not\in I \cup J$. Then, by Lemma~\ref{lem:outneighbor},
there exists a vertex $\mathbf{w}$ in $I \cup J$
such that $\mathbf{w} \prec \mathbf{v}$.
Then $\mathbf{w}\prec\mathbf{x}$ and $\mathbf{w}\prec\mathbf{y}$
and so $(\mathbf{x},\mathbf{w})$ and $(\mathbf{y},\mathbf{w})$ belong to $A(D_S)$.
Thus $H \cup (I \setminus V(H)) \cup J$ is a subgraph of the competition graph of $D_S$
and we have shown that it is the competition graph of $D_S$.
Hence $\dim_{\text{{\rm poc}}}(H) \leq d$ and the proposition holds.
\end{proof}

It does not seem to be easy to compute the partial order competition dimension of a graph in general.
In this context, we first characterize graphs having small partial order competition dimensions.
In such a way, we wish to have a better idea to settle the problem.

Let $K_n$ denote the complete graph with $n$ vertices.

\begin{Prop}\label{prop:dim0}
Let $G$ be a graph.
Then, $\dim_{\text{{\rm poc}}}(G)= 0$ if and only if $G = K_1$.
\end{Prop}

\begin{proof}
The proposition immediately follows from the definition of $0$-partial order.
\end{proof}

\begin{Prop}\label{prop:dim1}
Let $G$ be a graph.
Then, $\dim_{\text{{\rm poc}}}(G)= 1$ if and only if
$G = K_{t+1}$ or
$G = K_t \cup K_1$ for some positive integer $t$.
\end{Prop}

\begin{proof}
First we remark that
if $D$ is a $1$-partial order with $V(D) \subseteq \mathbb{R}^1$
and $v^* \in V(D)$ is the minimum among $V(D)$,
then the competition graph of $D$ is the disjoint union of
a clique $V(D) \setminus \{v^*\}$
and an isolated vertex $v^*$.
Therefore, if $\dim_{\text{{\rm poc}}}(G) = 1$,
then we obtain
$G = K_{t+1}$ or
$G = K_t \cup K_1$ for some nonnegative integer $t$.
By Proposition \ref{prop:dim0}, $G \neq K_{1}$
and thus $t$ is a positive integer.

If $G = K_{t+1}$ or
$G = K_t \cup K_1$ for some positive integer $t$,
then we obtain $\dim_{\text{{\rm poc}}}(G) \leq 1$.
By Proposition \ref{prop:dim0}, since $G \neq K_{1}$,
we have $\dim_{\text{{\rm poc}}}(G) = 1$.
\end{proof}

\begin{Lem}\label{lem:iso}
Let $G$ be a graph such that $\dim_{\text{{\rm poc}}}(G) \geq 2$ and
let $G'$ be a graph obtained from $G$ by adding isolated vertices.
Then
$\dim_{\text{{\rm poc}}}(G)= \dim_{\text{{\rm poc}}}(G')$.
\end{Lem}

\begin{proof}
Let $a_1, \ldots, a_k$ be the isolated vertices added to $G$ to obtain $G'$.
Let $d = \dim_{\text{{\rm poc}}}(G)$.
Then $G$ can be made into  the competition graph a $d$-partial order $D$
by adding sufficiently many isolated vertices.
Since $d \geq 2$, we can locate $k$ points $\mathbf{a}_1, \ldots, \mathbf{a}_k$ in $\mathbb{R}^d$
corresponding to $a_1, \ldots, a_k$ so that
no two points in $\{\mathbf{a}_1, \ldots, \mathbf{a}_k \}$
are related by $\prec$
and that
no point in $V(D)$ and no point in $\{ \mathbf{a}_1, \ldots, \mathbf{a}_k \}$
are related by $\prec$.
Indeed, we can do this in the following way:
for $i=1, \ldots, k$,
let $\mathbf{a}_i$ be a point in $\mathbb{R}^d$ defined by
\[
(\mathbf{a}_i)_1 = r_1 + i; \quad
(\mathbf{a}_i)_2 = r_2 - i; \quad
(\mathbf{a}_i)_j = 0 \ (j=3,\ldots,d),
\]
where
$r_1 := \max\{ (\mathbf{v})_1 \mid \mathbf{v} \in V(D) \}$
and
$r_2 := \min\{ (\mathbf{v})_2 \mid \mathbf{v} \in V(D) \}$.
Then $G'$ is the competition graph of $D$ together with $\mathbf{a}_1, \ldots, \mathbf{a}_k$
and thus $\dim_{\text{{\rm poc}}}(G') \leq \dim_{\text{{\rm poc}}}(G)$.

Since $G$ is an induced subgraph of $G'$, by Proposition~\ref{prop:induced},
$\dim_{\text{{\rm poc}}}(G) \leq \dim_{\text{{\rm poc}}}(G')$.
Hence $\dim_{\text{{\rm poc}}}(G) = \dim_{\text{{\rm poc}}}(G')$.
\end{proof}

\begin{Prop}\label{prop:dim2}
Let $G$ be a graph.
Then, $\dim_{\text{{\rm poc}}}(G) = 2$
if and only if
$G$ is an interval graph which is neither $K_s$ nor $K_t \cup K_1$
for any positive intergers $s$ and $t$.
\end{Prop}

\begin{proof}
Suppose that $\dim_{\text{{\rm poc}}}(G) = 2$.
By Theorem \ref{dpo},
$G$ is an interval graph.
By Propositions \ref{prop:dim0} and \ref{prop:dim1},
$G$ is neither $K_s$ nor $K_t \cup K_1$ for any positive intergers $s$ and $t$.

Suppose that
$G$ is an interval graph which is
neither $K_s$ nor $K_t \cup K_1$ for any positive intergers $s$ and $t$.
By Theorem \ref{interval},
$\dim_{\text{{\rm poc}}}(G) \leq 2$.
By Propositions \ref{prop:dim0} and \ref{prop:dim1}
$\dim_{\text{{\rm poc}}}(G) \ge 2$.
Thus, $\dim_{\text{{\rm poc}}}(G) = 2$.
\end{proof}

\begin{Prop}\label{prop:cycle}
If $G$ is a cycle of length at least four, then
$\dim_{\text{{\rm poc}}}(G) = 3$.
\end{Prop}

\begin{proof}
Let $G$ be a cycle of length $n$ with $n \geq 4$.
Note that $G$ is not an interval graph.
By Propositions \ref{prop:dim0}, \ref{prop:dim1}, and \ref{prop:dim2},
we have $\dim_{\text{{\rm poc}}}(G) \geq 3$.
Let $\mathcal{F}$ be the family of $n$ closed triangles given in Figure~\ref{dimcycle}.
\begin{figure}
\begin{center}
\includegraphics[height=5cm]{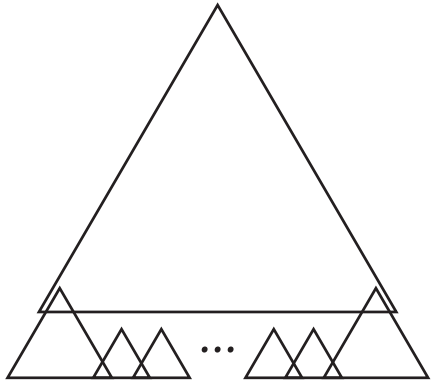}
\end{center}
\caption{A family of homothetic equilateral closed triangles}
\label{dimcycle}
\end{figure}
Then the intersection graph of $\mathcal{F}$ is the cycle of length $n$.
By Theorem~\ref{thm:closed} with $d=3$, $G$ together with sufficiently many isolated vertices
is the competition graph of a $3$-partial order.
Thus $\dim_{\text{{\rm poc}}}(G) \leq 3$.
Hence $\dim_{\text{{\rm poc}}}(G) = 3$.
\end{proof}

\begin{Thm}
If a graph $G$ contains an induced cycle of length at least four,
then $\dim_{\text{{\rm poc}}}(G) \geq 3$.
\end{Thm}

\begin{proof}
The theorem follows from Propositions~\ref{prop:induced} and~\ref{prop:cycle}.
\end{proof}

\begin{Thm}\label{thm:tree}
Let $T$ be a tree.
Then $\dim_{\text{{\rm poc}}}(T) \leq 3$,
and the equality holds if and only if $T$
is not a caterpillar.
\end{Thm}

\begin{proof}
By Theorem~\ref{thm:closed} with $d=3$, we need to show that
there exists a family of homothetic equilateral closed triangles in $\mathbb{R}^2$ whose intersection graph is $T$.
As a matter of fact,
it is sufficient to find such a family in the $xy$-plane
with the base of each triangle parallel to the $x$-axis.
We call the vertex of a triangle which is opposite to the base the \emph{apex} of the triangle.
We show the following stronger statement
by induction on the number of vertices:
\begin{quote}
For a tree $T$ and a vertex $v$ of $T$,
there exists a family $\mathcal{F}^T_v := \{ \triangle(x) \mid x \in V(T)\}$
of homothetic equilateral closed triangles whose intersection graph is $T$ such that,
for any vertex $x$ distinct from $v$, the apex
and the base of $\triangle(x)$ are below the apex and the base of $\triangle(v)$, respectively.
\end{quote}
We call the family $\mathcal{F}^T_v$ in the above statement
a \emph{good family} for $T$ and $v$.

If $T$ is the tree having exactly one vertex, then the statement is vacuously true.
Assume that the statement holds for any tree on $n-1$ vertices, where $n \geq 2$.
Let $T$ be a tree with $n$ vertices.
We fix a vertex $v$ of $T$ as a root.
Let $T_1, \ldots, T_k$ $(k \geq 1)$ be the connected components of $T - v$.
Then $T_1, \ldots, T_k$ are trees.
For each $i = 1, \ldots, k$,
$T_i$ has exactly one vertex, say $w_i$, which is a neighbor of $v$ in $T$.
We take $w_i$ as a root of $T_i$.
By the induction hypothesis,
there exists a good family $\mathcal{F}^{T_i}_{w_i}$ for $T_i$ and $w_i$
for each $i = 1, \ldots, k$.
Preserving the intersection or the non-intersection of two triangles in $\mathcal{F}^{T_i}_{w_i}$
for each $i=1, \ldots, k$,
we may translate the triangles in
$\mathcal{F}^{T_1}_{w_1} \cup \cdots \cup \mathcal{F}^{T_k}_{w_k}$
so that the apexes of $\triangle(w_1)$, $\ldots$, $\triangle(w_k)$ are on the $x$-axis
and any two triangles from distinct families do not intersect.
Let $\delta_i$ be the distance between the apex of $\triangle(w_i)$ and the apex of a triangle
which is the second highest among the apexes of the triangles in $\mathcal{F}^{T_i}_{w_i}$.
Now we draw a triangle $\triangle(v)$ in such a way that
the base of $\triangle(v)$ is a part of the line
$y = -\frac{1}{2}\min \{ \delta_1, \ldots, \delta_k \}$
and long enough to intersect all of the triangles $\triangle(w_1), \ldots, \triangle(w_k)$.
Then the family
$\mathcal{F}^{T}_{v} := \mathcal{F}^{T_1}_{w_1} \cup \cdots \cup \mathcal{F}^{T_k}_{w_k}
\cup \{\triangle(v)\}$ is a good family for $T$ and $v$
and thus the statement holds.
Hence, $\dim_{\text{{\rm poc}}}(T) \leq 3$ for a tree $T$.

Since trees which are interval graphs are caterpillars,
the latter part of the theorem follows from
Propositions \ref{prop:dim0}, \ref{prop:dim1}, and \ref{prop:dim2}.
\end{proof}

\section{Concluding Remarks}

In this paper, we studied the competition graphs of $d$-partial orders
and gave a characterization by using homothetic open simplices.
Since any graph can be made into the competition graph of a $d$-partial order
for some positive integer $d$ by adding isolated vertices,
we introduced the notion of the partial order competition dimension of a graph.
We gave characterizations of graphs having partial order competition dimension $0$, $1$, and $2$.
We also showed that cycles and trees have
partial order competition dimension at most $3$.
It would be an interesting research problem
to characterize graphs $G$ having partial order competition dimension $3$.



\begin{thebibliography}{99}

\bibitem{chokim}
{H.~H.~Cho and S.~-R.~Kim}:
{A class of acyclic digraphs with interval competition graphs},
\emph{Discrete Applied Mathematics}
\textbf{148} (2005) 171--180.

\bibitem{SJkim}
{S.~-J.~Kim, S.~-R.~Kim, and Y.~Rho}:
{On CCE graphs of doubly partial orders},
\emph{Discrete Applied Mathematics}
\textbf{155} (2007) 971--978.

\bibitem{Kim93}
{S.~-R.~Kim}:
{The competition number and its variants},
in J.~Gimbel, J.~W.~Kennedy, and L.~V.~Quintas (eds.),
\emph{Quo Vadis Graph Theory?},
\emph{Ann.\ Discrete Math.}, Vol. \textbf{55} (1993) 313--325.

\bibitem{Niche2009}
{S.~-R.~Kim, J.~Y.~Lee, B.~Park, W.~J.~Park, and Y.~Sano}:
{The niche graphs of doubly partial orders},
\emph{Congressus Numerantium}
\textbf{195} (2009) 19--32.

\bibitem{hypergraph}
{S.~-R.~Kim, J.~Y.~Lee, B.~Park, and Y.~Sano}:
{The competition hypergraphs of doubly partial orders},
\emph{Discrete Applied Mathematics}
\textbf{165} (2014) 185--191.

\bibitem{LuWu}
{J.~Lu and Y.~Wu}:
{Two minimal forbidden subgraphs for double competition graphs
of posets of dimension at most two},
\emph{Applied Mathematics Letters}
\textbf{22} (2009) 841--845.

\bibitem {Lundgren89}
{J.~R.~Lundgren}:
Food webs, competition graphs, competition-common enemy graphs, and niche graphs,
in F.S. Roberts (ed.),
\emph{Applications of Combinatorics and Graph Theory in the
Biological and Social Sciences},
IMA Volumes in Mathematics and its Applications, Vol. \textbf{17},
Springer-Verlag, New York (1989) 221--243.

\bibitem{PLK:mStepDPO}
{B.~Park, J.~Y.~Lee, S.~-R.~Kim}:
{The $m$-step competition graphs of doubly partial orders},
\emph{Applied Mathematics Letters}
\textbf{24} (2011) 811--816.

\bibitem{Phylogeny}
{B.~Park and Y.~Sano}:
{The phylogeny graphs of doubly partial orders},
\emph{Discussiones Mathematicae Graph Theory}
{\bf 33} (2013) 657--664.

\bibitem{dim2poset}
{Y.~Wu and J.~Lu}:
{Dimension-2 poset competition numbers and dimension-2 poset double competition numbers},
\emph{Discrete Applied Mathematics}
{\bf 158} (2010) 706--717.

\end{thebibliography}
\end{document}